\documentclass[a4paper,12pt,oneside]{article}

\usepackage{mathtools}
\usepackage{tikz}
\usetikzlibrary{shapes,arrows,positioning} \usepackage[framemethod=tikz]{mdframed}
\usepackage{soul}
\usepackage{transparent}
\usetikzlibrary{arrows, decorations.markings}
\usetikzlibrary{automata,positioning}
\usepackage[T1]{fontenc}
\usepackage{tikz}
\usetikzlibrary{shadings,shadows,shapes.arrows}
\usetikzlibrary{shapes,snakes}
\usetikzlibrary{matrix}
\usepackage{fancybox}
\usepackage{etex}
\usepackage{tcolorbox}
\usepackage{color}
\usepackage{fancybox,framed}

\usepackage[framemethod=tikz]{mdframed}
\usepackage{lipsum}
\usepackage{amssymb}
\usepackage{pifont}
\definecolor{mycolor}{rgb}{0.122, 0.435, 0.698}
\usepackage{emptypage}
\usepackage{afterpage}

\usepackage{epigraph}
\usepackage{url}
\usepackage{fancyhdr}
\usepackage{enumitem}
\usepackage{wrapfig}
\usepackage[toc,page]{appendix}
\usepackage{tikz}
\usepackage{fp}
\usepackage{pgf}
\usepackage{xifthen}

\usepackage[utf8]{inputenc}

\usepackage{color}
\usepackage{import}
\usepackage{setspace}
\usepackage{latexsym,amsmath,amsfonts,amscd,amssymb,amsthm}
\usepackage{graphicx, caption}
\usepackage[all]{xy}
\usepackage{epstopdf}
\DeclareGraphicsExtensions{.pdf,.png,.jpg,.gif, .eps}

\usepackage{amsthm}
\usepackage{thmtools}

\theoremstyle{plain} 
\newtheorem{teo}{Theorem}
\newtheorem{lem}{Lemma}
\newtheorem{prop}{Proposition}
\newtheorem{cor}{Corollary}
\theoremstyle{definition}

\theoremstyle{remark}
\newtheorem{obs}{Remark}
\newtheorem{ej}{Example}

\newcommand{\parentesis}[1]{\left(#1\right)}

\newcommand{\conjunto}[1]{\left\lbrace #1 \right\rbrace}
\newcommand{\mapeo}[5]{
	\begin{eqnarray*}
		#1:#2 & \longrightarrow & #3\\
		#4 & \longmapsto & #5
\end{eqnarray*}}
\newcommand{\cms}{$(X,\textrm{d})$ }

\newcommand{\todon}{n\in\mathbb{N}}

\renewcommand{\epsilon}{\varepsilon}

\newcommand{\diam}{\textrm{diam}}
\newcommand{\hiper}{2^{X_d}_u}
\newcommand{\hiperf}{2^{X_d}_f}

\newcommand{\hipernf}{2^{\mathbb{N}}_f}
\newcommand{\texfor}[1]{\enspace\textrm{#1}\enspace}
\newcommand{\comp}[2]{\left(#1\right)_{#2}}

\usepackage{hyperref}
\hypersetup{
	colorlinks=true,
	linkcolor=blue,
	filecolor=magenta,      
	urlcolor=cyan,
	pdfpagemode=FullScreen,
}

\usepackage[light, math]{iwona}
\usepackage[T1]{fontenc}
\makeindex
\makeatletter
\newcommand{\subjclass}[2][2010]{%
	\let\@oldtitle\@title%
	\gdef\@title{\@oldtitle\footnotetext{#1 \emph{Mathematics Subject Classification.} #2}}%
}
\newcommand{\keywords}[1]{%
	\let\@@oldtitle\@title%
	\gdef\@title{\@@oldtitle\footnotetext{\emph{Key words and phrases.} #1.}}%
}
\makeatother
\title{\bf{A universal space for finite topological spaces}}
\usepackage{authblk}
\author{Diego Mondéjar}
\date{}                     

\keywords{Universal spaces, Alexandroff and finite spaces, hyperspaces, upper semifinite topology, inverse limits}
\subjclass{Primary 54B20, 54C25; Secondary 54A10, 54A05, 54D10, 54D35}

\begin{document}
	
\maketitle

	\begin{abstract}
		We find universal spaces for Alexandroff and finite spaces and explore some of its topological properties as well as their description as inverse limits of finite spaces and Alexandroff extensions. They can be used as a natural environment to describe shape properties of compact metric spaces.
	\end{abstract}

\section{Introduction}
Finite and, more generally, Alexandroff spaces are getting more attention in topology research because of their use and applications in the fields of Applied, Computational and Digital Topology. Alexandroff \cite{Adiskrete} first noticed that Alexandroff $T_0$ spaces were essentially \textsc{poset}s. Their topological properties are very special and they were not very studied during some time. Although they have only a finite set of points, they are far from discrete and enjoy interesting topological properties. According W. Thurston \cite{Tonp}, the study of finite topological spaces is "an oddball topic that can lend good insight to a variety of questions". Stong \cite{Sfinite} and McCord \cite{Msingular} found interesting properties about its homeomorphism and homotopy types, and stablished a deep conection with simplicial complexes, showing that they have the same weak homotopy types. In May notes \cite{Mfinitetopological,Mfinitecomplexes}, these results are presented in a common framework, deriving in questions, conjectures and promising future work. In this context, Barmak and Minian \cite{BMsimple,BMautomorphism,Balgebraic} developed a whole algebraic topology theory for finite and Alexandroff spaces (see also Kukie{\l}a \cite{Kon} and Hardie and Vermeulen \cite{HVhomotopy} works about the homotopy of general Alexandroff spaces). Hyperspaces are a way of building new topological spaces out from one given as the set of closed subsets with different possible topologies (see \cite{Nhyperspaces}). In this paper we focus on finding universal spaces for Alexandroff spaces and, specially, for some finite topological spaces and for this we use hyperspaces. 

The motivation of this paper is twofold. On the one hand, in \cite{MGupper, MGhomotopical, MGthe}, Morón and Gómez uses embeddings for every Tychonov space in its hyperspace with the upper semifinite topology. They relate properties of the space with properties of the hyperspace using that, although the topology of the hyperspace is non-Hausdorff, it is very easy to manipulate. They also describe a special neighborhood system for the embedding of a compact metric space in its upper semifinite hyperspace to get results in the shape theory for compacta. On the other hand, in \cite{MMreconstruction}, the authors showed (Theorem 4) that every compact metric space is homeomorphic to a subspace of the inverse limit of finite topological spaces (actually, upper semifinite hyperspaces of discrete approximations), being the inverse limit an homotopical copy of the space with a homeomorhic copy as a strong deformation retract. The generalization for the hyperspace is also true, namely that the hyperspace of the compact metric space with the upper semifinite topology is homotopically equivalent to an inverse sequence of finite topological hyperspaces (Theorem 20 of \cite{phMondejar}). Also, in \cite{Mpolyhedral}, it is shown that the associated inverse sequence of polyhedra (by means of the Alexandroff-McCord functor \cite{Msingular}) represents the shape type of the original space. Combining this results, we try to answer some related questions. What classes of spaces can be embedded in hyperspaces of discrete spaces with the upper semifinite topology? Can these spaces be written as inverse limits of finite spaces? Can we use them as a universal environment to detect shape properties?

We begin with Section \ref{sec:Hyperspaces}, where we define hyperspaces of discrete ones that are Alexandroff and $T_0$. Then, we consider the subspace of the previous one consisting of finite elements and called the power of finite sets. It is shown in Proposition \ref{prop:contractible} that it is contractible and we determine its homeomorphisms in Proposition \ref{prop:bijection}. Its Alexandroff extension is a subspace of a hyperspace (Proposition \ref{prop:extension}) and can be written as an inverse limit of finite hyperspaces, which is shown in Theorem \ref{teo:inverseextension}. In Section \ref{sec:embeddings}, we show in Proposition \ref{prop:alexembeddings} that the power of finite sets is able to embed every $T_0$ Alexandroff space (and hence $T_0$ finite) that is locally finite and, as Corollary \ref{cor:finiteuniversal} we find that the power of finite sets os natural numbers is universal for finite $T_0$ spaces. We use the Alexandroff-Mc-Cord correspondence to embed spaces with the same weakly homotopy type that an Alexandroff one as an open subset in these hyperspaces in Proposition \ref{prop:weakly}. Section \ref{sec:simplicial} is devoted to define a category equivalent (Corollary \ref{cor:equivalent}) to the category of abstract simplicial complexes and simplicial maps, using the previous results. Finally, we propose a shape type description using subspaces of the power of finite sets, in Proposition \ref{prop:univshape}, making it universal for shape of compact metric spaces.
\section{Hyperspaces of discrete spaces}\label{sec:Hyperspaces}
An \emph{Alexandroff space} $X$ (sometimes called A-space) is one in which every intersection of open sets is open, hence they have \emph{minimal neighborhoods} for every point, written $B_x$ for every $x\in X$. Finite topological spaces are obviously Alexandroff. We recommend \cite{Mfinitetopological, Mfinitecomplexes} and the main references for Alexandroff and finite topological spaces. The set of minimal neighborhoods $B_x$ for every point of the space $x\in X$ is a base for the topology called the \emph{minimal basis}. This minimal basis defines a reflexive and transitive relation on the space $X$. For $x,y\in X$, say $x\leqslant y$ if $B_x\subset B_y$. On the other hand, every reflexive and transitive relation on a set $X$ determines an Alexandroff topology, with basis the sets $U_x=\{y\in X:y\leqslant x\}$. So, we have that Alexandroff topologies in a set $X$ are in bijective correspondence with its reflexive and transitive relations. The topology is $T_0$ if and only if the relation is a partial order. We call a set with a partial order a \caps{poset}, so Alexandroff $T_0$ spaces and \caps{poset}s are the same thing. In what follows we will use both points of view without distinction. With this notation, continuous maps between them are easily characterized. A function $f:X\rightarrow Y$ of Alexandroff spaces is continuous if and only if is order preserving, that is, $x\leqslant y$ implies $f(x)\leqslant f(y)$. Although they can have interesting topological structure, their separation properties are poor: an Alexandroff space that is $T_1$ is discrete. From the works of Stong \cite{Sfinite} and \cite{Msingular}, we know that every Alexandroff space is homotopically equivalent to an Alexandroff $T_0$ space so, up to homotopy, we can reduce our study to them. Also, we can relate simplicial complexes with Alexandroff $T_0$ spaces by the \emph{Alexandroff-McCord correspondence} between them, which states than they are homotopically weakly equivalent. We will review it later.

For every topological space $X$, let $2^X$ be the set of non-empty closed subsets of $X$, called the \emph{hyperspace} of $X$. The \emph{upper semifinite topology} in $2^X$ is generated by the base $$B(U)=\{C\in2^X: C\subset U\},\enspace\text{$U$ open in $X$}$$ and we write $2^X_{u}$ for the hyperspace of $X$ with the upper semifinite topology \cite{Nhyperspaces}.

First of all, we give the following result about hyperspaces of Alexandroff spaces with this topology.
\begin{prop}
	For every Alexandroff space $X$, the hyperspace $2^X_{u}$ is an Alexandroff space.
\end{prop}
\begin{proof}
	As an Alexandroff space, every point $x\in X$ has a minimal neighborhood $B_x$. Now, consider a point $C\in 2^X_u$. Assume that $C=\{x_j\}_{j\in J}$, where $x_j\in X$. For every $j\in J$, consider the open neighborhood $B(B_{x_j})$ in $2^X_u$. Then we claim that $$B_C=\bigcup_{j\in J}B(B_{x_j})$$ is the minimal neighborhood of $C$. Consider any basic open neighborhood $B(U)$ of $C$. Then, $C\subset U$, so $x_j\in U$ for every $j\in J$, and then, $B_{x_j}\subset U$. Then it is clear, for every $j\in J$, that $B(B_{x_j})\subset U$ and hence $B_C\subset B(U))$.
\end{proof}
\begin{ej}
	The converse of this proposition is not true, even for $T_0$ Alexandroff spaces, as shown in the following example: Consider the unit interval $I=[0,1]$ with the topology having as proper open sets the half intervals $[0,t)$ with $t\in(0,1)$. Now consider the subspace $X=\{\frac{1}{n}:n\in\mathbb{N}\}\cup\{0\}$ with the subspace topology. It is a $T_0$ non-Alexandroff space, since the (of course infinite) intersection of open sets $$\bigcap_{t\in(0,1)}\left([0,t)\cap X\right)=\{0\}$$ is not an open set. But, it turns out that the hyperspace $2^X_u$ is Alexandroff. Every proper closed set is of the form $X\backslash[0,t)$, with $t\in(0,1)$, i.e., $X_n=\conjunto{\frac{1}{n},\frac{1}{n-1},\ldots,1}$ with $\todon$. For every $\todon$, the only open set containing $X_n$ is $X$, so every point of $2^X_u$ has $B(X)=2^X_u$ as minimal neighborhood.
\end{ej}

Let $X_d$ be a set with the discrete topology. We consider the hyperspace of non-empty subsets with the upper semifinite topology $\hiper$. If $X$ is discrete, then it is an Alexandroff space, and we already know that the hyperspace is Alexandroff. We prove it explicitly for this case, in order to find the minimal basis and we also show that the hyperspace is also $T_0$.
\begin{prop}
	For every set $X$, the space $\hiper$ is a $\text{T}_0$ (and not necessarily $T_1$) Alexandroff space.
\end{prop}
\begin{proof}
	To show that $\hiper$ is Alexandroff, we need to find, for every point $C$ of the space, a minimal neighborhood $B_C$. Every subset of $X$ is open with the discrete topology, so the basis element associated to the open set $C$ is $B(C)=2^C$. It is easy to see that $$B_C=\bigcap_{C\subset U\text{ open in } X}B(U)=B(C),$$ so it contains $C$ and it is contained in any open neighborhood of $C$. Hence $\hiper$ is an Alexandroff space with minimal neighborhoods $2^C$ for every $C\in\hiper$. It is $\text{T}_0$ because, for every pair of different points $C,D\in\hiper$, there are two possibilities. If $C\varsubsetneq D$, then $C\in B_C\not\ni D$. If $C\not\subset D$ and $D\not\subset C$ then we have both $C\in B_C\not\ni D$ and $C\not\in B_D\in D$. The space $\hiper$ is not generally $\text{T}_1$ because for every two different points with $C\subset D$, every neighborhood of $D$ contains $B_D\supset B_C\ni C$.
\end{proof}
\begin{obs}
	As a $T_0$ Alexandroff space, the partial order induced in $\hiper$ is the following: for every $C,D\in\hiper$, $C\leqslant D$ if and only if $C\subseteq D$. Moreover, the minimal neighborhood for a point $C\in\hiper$ is $2^C$.
\end{obs}

We now consider the \emph{power of finite sets of $X$}, that is, $$\hiperf=\conjunto{C\in\hiper:\text{card}(C)\enspace\text{is finite}}\subset\hiper$$ which receives the subspace topology, so $$\conjunto{B(U)\cap\hiperf:\text{$U$ open in $X$}}$$ is a basis for its topology. As a subspace, $\hiperf$ is a $T_0$ Alexandroff space, with minimal neighborhoods $2^C$,  for each $C\in\hiperf$. Define, for every $r\in\mathbb{N}$, the set $$2^{X_d}_r=\conjunto{C\in 2^{X_d}_f:\text{card}(C)\leqslant r}$$ of points of the hyperspace $\hiper$ with a bounded (by $r$) number of elements.
\begin{obs}
	For every $r\in\mathbb{N}$, the space $2^{X_d}_r$ is an open subset of $2^{X_d}_f$. This is so because, for every $C\in 2^{X_d}_r$, its minimal neighborhood $2^C$ is contained in the space $2^{X_d}_r$. Moreover, for every pair $r\leqslant s$, we have the inclusion $2^{X_d}_r\subset 2^{X_d}_s$.
\end{obs}
\begin{ej}In general, the space $\hiperf$ is not $T_1$. For example, for the set of natural numbers $\mathbb{N}$, we have that $\hipernf$ is not $T_1$ because, for example, the minimal neighborhood of $\{1,2,3\}$ contains $\{1,2\}$. 
\end{ej}

We study some topological properties of these hyperspaces. 

A subspace $A$ of $X$ is a \emph{deformation retract} of $X$ if there is a map (homotopy) $H:X\times I\rightarrow X$ such that for all $x\in X$, $a\in A$ and $t\in I$ the map $H$ satisfies $H(x,0)=x$ and $H(x,1)\in A$. The deformation retract is also called \emph{strong} if, moreover, $H(a,t)=a$. A topological space is \emph{contractible} if it has the homotopy type of a point. Deformation retracts preserve homotopy, that is, if a point is a deformation retract of a space, it is contractible.

\begin{prop}\label{prop:contractible}
For every set $X$, every point of the space $\hiperf$ is a strong deformation retract of it and hence it is contractible.
\end{prop}
\begin{proof}Let $A$ be any point of $\hiperf$. We will prove that $\hiperf$ can be retracted to $A$. The map $H:\hiperf\times[0,1]\rightarrow\hiperf$ defined by 
	\begin{equation*}
		H(C,t) = \left\{
		\begin{array}{rl}
			C & \text{if } t\in[0,\frac{1}{2}),\\
			C\cup A & \text{if } t=\frac{1}{2},\\
			A & \text{if } t\in(\frac{1}{2},1]
		\end{array}\right.
	\end{equation*}
	is an homotopy satisfying $H(C,0)=C$ and $H(C,1)=A$, so $A$ is a strong deformation retract of $\hiperf$. Indeed, we shall prove that this map is continuous everywhere. Let $(C,t)\in\hiperf\times[0,1]$.
	\begin{itemize}
		\item If $t\in\left[0,\frac{1}{2}\right)$ then $H(C,t)=C$. Let $V$ be a neighborhood of $C$, we know that $C\in B_C\subset V$. The neighborhood $U=\hiperf\times\left[0,\frac{1}{2}\right)$ of $(C,t)$ satisfies $H(U)=C\in V$. 
		\item If $t\in\left(\frac{1}{2},1\right]$ then $H(C,t)=C$. Let $V$ be any neighborhood of $A$, we have that $A\in B_A\subset V$. The neighborhood of $(C,t)$ given by $U=\hiperf\times\left(\frac{1}{2},1\right]$ satisfies $H(U)=A\in V$.
		\item Finally, $H(C,\frac{1}{2})=C\cap A$. For any neighborhood $V$ of $C\cap A$ we can claim that $C\cap A\in B_{C\cap A}\subset V$ so the image of the neighborhood of $(C,\frac{1}{2})$ given by $U=\hiper\times[0,1]$ satisfies $H(U)=C\cap A\in V$.
	\end{itemize}
\end{proof}
This is quite non-evident since this space is highly non-homogeneous as we will see. A topological space $X$ is said to be \emph{homogeneous} if, for every two points $x,y\in X$, there is a homeomorphism $f:X\rightarrow X$ such that $f(x)=y$. In other words, the group of self homeomorphisms of $X$ is transitive in $X$.

We next characterize the homeomorphisms of our space in order to measure its inhomogeneity. To do so, we use the \emph{elevation} $2^g$ of maps between topological spaces $g:X\rightarrow Y$ to hyperspaces (with any topology), defined as $2^g:2^X\rightarrow 2^Y$ with $2^g(C)=\overline{g(C)}$. Note that, if $X,Y$ are discrete, we have can define it merely as $2^g(C)=g(C)$. The following result is stated in Proposition 2.8 of \cite{MGupper}.
\begin{prop}\label{prop:homeos}
	A map between Tychonov spaces $g:X\rightarrow Y$ is a homeomorphism if and only if so it is the corresponding elevation between the corresponding hyperspaces with the upper semifinite topology $2^g:2^X\rightarrow 2^Y$.
\end{prop} 
In our case, the homeomorphisms correspond to bijections between the spaces.
\begin{prop}\label{prop:bijection}
	Let $X$ be any set. Then a function $f:\hiperf\rightarrow\hiperf$ is a homeomorphism if and only if there exists a bijection $\gamma:X\rightarrow X$ such that $f=2^{\gamma}$. That is, the homeomorphism group of $\hiperf$ is isomorphic to the group of permutations of $\text{card}(X)$ elements.
\end{prop}
\begin{proof}
	This is a direct consecuence of Proposition \ref{prop:homeos},  because every set with the discrete topology is a Tychonov space. In this particular situation, the proof is simpler, as shown.
	
	Consider we have a bijection $\gamma$ of $X$ and define $f:\hiperf\rightarrow\hiperf$ as the elevation $f=2^{\gamma}$. Then, for every $C\in\hiperf$, we have $f(C)=\cup_{c\in C}\gamma(c)$. The map $f$ is continuous and open, since, for every $C\in\hiperf$, we have $$f(2^C)=\bigcup_{D\subset C}\parentesis{\bigcup_{c\in D}\gamma(c)}=2^{\bigcup_{c\in C}\gamma(c)}=2^{f(C)}.$$ It is clearly injective. If $C\neq D$, let us suppose that there exists $d\in D\backslash C$. Then $\gamma(d)\in f(D)\backslash f(C)$. Finally, $f$ is surjective: For every $C\in\hiperf$, $C=\bigcup_{c\in\gamma^{-1}(C)}\gamma(c)=f(\gamma^{-1}(C))$. We conclude that $f$ is a homeomorphism. On the other hand, let $f:\hiperf\rightarrow\hiperf$ be a homeomorphism. Consider a point of $\hiperf$ consisting of only one point of $X$, that is $\{x\}\in\hiperf$. Let us write $f(\{x\})=C\in\hiperf$. Then, $f^{-1}$ is a continuous map sending $C$ to $\{x\}$, so $f^{-1}(2^C)\subset2^{\{x\}}=\{x\}$. But $f^{-1}$ must be injective so $\text{card}\parentesis{2^C}=1$, hence $C=\{y\}$ with $y\in X$. That means there exists a function $\gamma:X\rightarrow X$ such that $f(\{x\})=\gamma(x)$ for every $x\in X$. This function must be a bijection, since $f$ is. Now, let us consider $C\in\hiperf$. Since $f$ is continuous, we have $\gamma(c)=f(\{c\})\subset f(C)$ for every $c\in C$, that is, $\bigcup_{c\in C}\gamma(c)$. On the other hand, since $f^{-1}$ is continuous, for every $d\in f(C)$ we have $\gamma^{-1}(d)=f^{-1}(\{d\})\subset C$, and hence $d\in\bigcap_{c\in C}\gamma(c)$. We conclude $f(C)=\bigcap_{c\in C}\gamma(c)$, i.e., $f=2^{\gamma}$.
\end{proof}
As an immediate corollary we obtain 
\begin{cor}
	Let $X$ be any set and consider $C,D\in\hiperf$. Then there exists a homeomorphism $f:\hiperf\rightarrow\hiperf$ with $f(C)=D$ if and only if $\text{card}(C)=\text{card}(D)$.
\end{cor}
\begin{proof}
	If $f:\hiperf\rightarrow\hiperf$ is a homeomorphism then, by the previous proposition, there exists a bijection $\gamma:X\rightarrow X$ such that $f=2^{\gamma}$. It is clear that the elevation of a bijection must preserve the cardinal of the elements.
	
	The opposite implication is straightforward because it is always possible to extend bijections to sets with the same cardinal. If $\text{card}(C)=\text{card}(D)$, then we can define two bijections $\alpha:C\rightarrow D$ and $\beta:X\backslash C\rightarrow X\backslash D$. Now we can define a bijection $\gamma:X\rightarrow X$ on the whole set by
	\begin{equation*}
		\gamma(x)= \left\{
		\begin{array}{rl}
			\alpha(x) & \text{if } x\in C,\\
			\beta(x) & \text{if } x\in X\backslash C.
		\end{array}\right.
	\end{equation*}
	Since it is a bijection, then $f=2^{\gamma}$ is a homeomorphism sending $C$ to $D$, as required.
\end{proof}
\begin{obs}\label{rem:bij}
	From the previous proof we can deduce that there exist exactly $\text{card}(C)!\cdot\text{card}(X\backslash C)!$ different homemorphisms (the combination of possible choices for the bijections $\alpha$ and $\beta$) sending $C$ to $D$.
\end{obs}
\begin{obs}
	The last proposition and its corollary remain true if we replace $\hiperf$ with $\hiper$. Nothing in the proofs actually changes.
\end{obs}
Hence, by Proposition \ref{prop:bijection} and Remark \ref{rem:bij}, we have that the only relevant information of the set $X$ that is kept in the hyperspaces $\hiper$ and $\hiperf$ is the cardinality. That is, we have the following
\begin{teo} Let $X,Y$ be sets with cardinalities $\omega_X$ and $\omega_Y$ respectively. Then, the following are equivalent:
	\begin{itemize}[noitemsep,nolistsep]
		\item [(i)] $\omega_X=\omega_Y$.
		\item [(ii)] $\hiper$ is homeomorphic to $2^{Y_d}_u$.
		\item [(iii)] $\hiperf$ is homeomorphic to $2^{Y_d}_f$.
	\end{itemize}
\end{teo}
Despite this fact, we will use the set notation instead of dealing just with cardinalities for the sake of simplicity.

Local finiteness is a property of topological spaces closely related to Alexandroff spaces. A topological space $X$ is called \emph{locally finite} if, for every $x\in X$, there exists a finite neighborhood $x\in U\subset X$. We will say that an Alexandroff space $A$ is \emph{strongly locally finite} (defined in \cite{Kon}) if, for every $a\in A$, the set of points related to $a$, that is, $$\conjunto{b\in A:b\leqslant a \text{ or } a\leqslant b},$$ is finite (equivalently, for every $a\in A$, $B_a$ and $\overline{\{a\}}^A$ are finite sets).

\begin{obs}Finite topological spaces are always locally finite. For Alexandroff spaces, strong local finiteness implies local finiteness.
\end{obs}
\begin{obs}
	For every infinite set $X$, the hyperspace $\hiperf$ is not finite, is locally finite but it is not strongly locally finite: For every $C\in\hiperf$, the minimal neighborhood $2^C$ is a finite open set containing $C$. But $C$ is contained in an infinite number of elements of $\hiperf$.
\end{obs}
It turns out that locally finite spaces are nothing but a special class of Alexandroff spaces.
\begin{prop}\label{teo:localesalexandroff}
	Every topological space is a locally finite space if and only if it is an Alexandroff space with finite minimal neighborhoods. 
\end{prop}
\begin{proof}
	Let $X$ be a locally finite space. Let us consider, for $x\in X$, a finite open neighborhood $x\in U\subset X$. We claim that $$B_U=\bigcap_{x\in B\subset U\texfor{open}}B$$ is the minimal open neighborhood for $x$. Note that it is not empty, because $B_U\subset U\ni x$, and open, because it is a finite intersection (remember $U$ is finite) of open sets. It is the minimal neighborhood of $x$ because, if $V$ is another open neighborhood of $x$, then $V\cap U\subset U$ is an open neighborhood of $x$, and hence $B_U\subset V$. Finally, the construction does not depend on the choice of the finite open neighborhood of $x$. If we use a different one, say $U'$, then $U\cap U'\subset U,U'$, so $B_U\subset U'$ and $B_{U'}\subset U$ which implies, respectively, that $B_{U'}\subset B_U$ and $B_U\subset U'$, so $B_U=B_{U'}=B_x$ is well defined. The converse is obviously true.
\end{proof}
Compactness and paracompactness are easily characterized in these hyperspaces:
\begin{prop}
	Let $X$ be any set. The following statements are equivalent:
	\begin{itemize}[noitemsep,nolistsep]
		\item[(i)] $X$ is finite.
		\item[(ii)]  $\hiperf$ is compact.
		\item[(iii)] $\hiperf$ is paracompact.
		\item[(iv)] $\hiperf$ is strongly locally finite.
	\end{itemize}
\end{prop}
\begin{proof} The implications (i)$\Rightarrow$(ii)$\Rightarrow$(iii) are obvious.
	\begin{itemize}
		\item (iii)$\Rightarrow$(iv) If $\hiperf$ is paracompact, then the minimal open covering $$\conjunto{2^C:C\in\hiperf}$$ must be locally finite, so, for every $C\in\hiperf$, $C\subset D$ for a finite number of points $D\in\hiperf$ (or, in other words, for every $C\in\hiperf$, the closure $\overline{\{C\}}^{\hiperf}$ is finite). Since $2^C$ is always finite too, $\hiperf$ space is strongly locally finite.
		\item (iv)$\Rightarrow$(i) If $X$ was infinite then, for every $C\in\hiperf$, we would have that $X\backslash C$ would be infinite, and then $C\subset C\cup D$ for every $D\in X\backslash C$, making $\overline{\{C\}}^{\hiperf}$ infinite, which is impossible.
	\end{itemize}
\end{proof}
We now construct the Alexandroff extension of $\hiperf$. It is similar to the process of a compactification and they exist for every topological space, although sometimes this name is reserved for the compactification of Hausdorf spaces. We use the general notion, following \cite{Kgen}. Let $X$ be any topological space and $\infty$ any point not in $X$. Consider the set $X^*=X\cup\{\infty\}$ with open sets the open sets of $X$ and the subsets $U\subset X^*$ such that $X^*\backslash U$ is closed and compact in X. The topological space $X^*$ is called the \emph{Alexandroff extension (or one-point compactification)} of $X$.

\begin{prop}[Alexandroff]
	The Alexandroff extension $X^*$ of a topological space $X$ contains it as a subspace and is compact. Moreover, $X^*$ is Hausdorff if and only if $X$ is locally compact and Hausdorff.
\end{prop}
There is another way of extending a topological space with one point. We follow \cite{Mfinitetopological} here. Let $X$ be a topological space and $*$ any point not in $X$. The \emph{non-Hausdorff cone} is the space $X\cup\{*\}$ with proper open sets the open sets of $X$.

\begin{obs}
	In general, for every topological space $X$, the topology of the Alexandroff extension is finer that the one in the non-Hausdorff cone.
\end{obs}
In order to find the Alexandroff extension of our space, we need the following lemma.
\begin{lem}
	If $X$ is an infinite set, there are no closed and compact subsets of $\hiperf$.
\end{lem}
\begin{proof}
	Consider a non-empty subset $B\subset\hiperf$ and suppose it is closed and compact. Consider a point $a\in B$, then $\overline{\{a\}}\subset\overline{B}=B$, being a closed subset of a compact space, is compact. But this is not possible: Consider the open covering $\bigcup_{\{a\}\subset C}2^C$ of $\overline{\{a\}}$, and suppose there is a finite subcovering, say $\conjunto{2^{C_1},\ldots,2^{C_s}}$. Then, for every $D\in C\backslash X$, we have $\{a\}\cup D\in\overline{\{a\}}$ but $\{a\}\cup D\notin\conjunto{2^{C_1},\ldots,2^{C_s}}$, so there are no possible finite subcoverings.
\end{proof}
\begin{obs}
	It turns out that given any set $X$, the Alexandroff extension and the non-Hausdorff cone of the hyperspace $2^X_f\subset 2^X_u$ are exactly the same topological space.
\end{obs}
We can identify a point in the hyperspace $\hiper$ so the Alexandroff extension of $\hiperf$ is just a subspace of it.
\begin{prop}\label{prop:extension}
	The subspace $\hiperf\cup\{X\}\subset\hiper$ is the Alexandroff extension of $\hiperf$.
\end{prop}
\begin{proof}
	We will show that they have exactly the same open sets. As a subspace, $\hiperf\cup\{X\}$ has as open sets the intersections with open sets of $\hiper$. Let $U$ be an open set of $\hiper$, and consider its intersection with $\hiperf\cup\{X\}$.\begin{itemize}
		\item If $X\notin U$, then the intersection is $U\cap\hiperf$.
		\item If $X\in U$, then $U$ has to be $2^{X}$, so the intersection is the whole set $\hiperf\cup\{X\}$.
	\end{itemize}
	On the other hand, as the Alexandroff extension, the open sets are the open sets of $\hiperf$ and the sets $V$, containing $X$, such that $\hiperf\backslash V$ is closed and compact. But, by the previous lemma, the only possibility is the empty set, so there is only one open set more, $\hiperf\cup\{X\}$.
\end{proof}
Now we give a description of the Alexandroff extension in terms of an inverse limit of subspaces of $\hiperf$. The definition of inverse limit of an inverse system (or sequence) can be found in \cite{MSshape,HYtopology}.
\begin{teo}\label{teo:inverseextension}
	Let $X$ be any set. Consider the hyperspace of finite subsets $\hiperf\subset\hiper$ with the upper semifinite topology. The Alexandroff extension of $\hiperf$ is homeomorphic to an inverse limit of an inverse system of finite spaces.
\end{teo}
\begin{obs}
	The proof of this theorem for the case $\hipernf$ (or, equivalently, when the cardinal of $X$ is countable) is simpler and more intuitive. Even the statement of the theorem we want to prove is then easier, because we just need a sequence (instead of a system) of finite spaces.  We include it in Appendix \ref{app:countable} and we recommend the reader to check this proof in order to understand the general case.
\end{obs}

We proceed with the general case.
\begin{proof}[Proof of Theorem \ref{teo:inverseextension}]
	We define first the inverse system. Consider the directed set $\hiperf$ in which $C\leqslant C'$ if $C\subset C'$. As objects, we consider the finite spaces $2^C$, for every $C\in \hiperf$. For every pair $C\leqslant C'$, define the map $p_{C,C'}:2^{C'}\rightarrow 2^C$ as 
	\begin{equation*}
		p_{C,C'}(D) = \left\{
		\begin{array}{lr}
			D & \text{if }D\subset C\\
			C & \text{if } D\not\subset C
		\end{array}\right.
	\end{equation*}
	for every $D\subset C'$. This map is continuous, because, for every pair $D\subset D'$ such that $D,D'\subset C'$, we have three possibilities. First, if $D\subset D'\subset C$, then $h(D)=D\subset D'=h(D')$. Second, if $D\subset C$ but $D'\not\subset C$, then $h(D)=D\subset C=h(D')$. Finally, if $D,D'\not\subset C$, then $h(D)=C=h(D')$.
	Let us write $\mathcal{X}$ for the inverse limit of the inverse system $\conjunto{2^C,p_{C,C'},\hiperf}$. We define a map
	$$h: \hiperf\cup\{X\}\longrightarrow\mathcal{X}$$ as follows. For $X\in2^X\cup\{X\}$, $\left(h(X)\right)_C=C$, for every $C\in\hiperf$. For every $D\in2^X\cup\{X\}$,
	\begin{equation*}
		\left(h(D)\right)_C = \left\{
		\begin{array}{lr}
			D & \text{if }D\subset C\\
			C & \text{if } D\not\subset C
		\end{array}\right.
	\end{equation*}
	We will show that $h$ is a homeomorphism and, in order to do that, we need to show several things.
	\begin{enumerate}
		\item $h$ is a well defined map. This is almost trivial in the case of the image of $X$, because it is $\comp{h(X)}{C}=C$ , for every $C\in \hiperf$, and $p_{C,C'}(C')=C$ for every $C\leqslant C'$. So $h(X)\in\mathcal{X}$. For  every $D\in\hiperf$, we have that, for every pair $C\leqslant C'$,
		\begin{equation*}
			p_{C,C'}\left(\left(h(D)\right)_{C'}\right) = \left\{
			\begin{array}{lr}
				p_{C,C'}(D) & \text{if }D\subset C'\\
				p_{C,C'}(C') & \text{if } D\not\subset C'
			\end{array}\right.=\left\{
			\begin{array}{lr}
				D & \text{if }D\subset C\\
				C & \text{if } D\not\subset C\\
				C & \text{if } D\not\subset C'
			\end{array}\right.=\comp{h(D)}{C},
		\end{equation*}
		hence $h(D)\in\mathcal{X}$.
		\item $h$ is continuous. The only possible open neighborhood of $h(X)$ is $\mathcal{X}$. For every $D\in\hiperf$, let us consider an open neighborhood $h(D)\subset W$ in $\mathcal{X}$. There exists an open set $V\subset\prod_{C\in2^X_F}2^C$ such that $h(D)\in V\cap\mathcal{X}\subset W$ with
		\begin{equation*}
			\comp{V}{C} = \left\{
			\begin{array}{ll}
				2^D & \text{if } C\in\conjunto{C_1,\ldots,C_n}\\
				2^C & \text{if not,}
			\end{array}\right.
		\end{equation*} 
		for some $\todon$ and with $D\subset C_1,\ldots,C_n$. We claim that $h(2^D)\subset V\cap\mathcal{X}$: For every $A\in 2^D$, we have that, if $C\in\{C_1,\ldots,C_n\}$, then $A\subset D\subset C$, so $\comp{h(A)}{C}=A\in 2^D$. If not, $\comp{h(A)}{C}$ could be $A$ or $C$, but both are in $2^C$.
		\item $h$ is surjective. Let $A$ be an element of $\mathcal{C}$. If, for every $C\in\hiperf$, $\comp{A}{C}=C$, then $h(X)=A$. If, on the contrary, there exists $C\in\hiperf$ such that $\comp{A}{C}=D\varsubsetneq C$, then we compute the rest of the projections as follows.
		\begin{itemize}
			\item For every $C'\leqslant C$, 
			\begin{equation*}
				p_{C',C}(D)= \left\{
				\begin{array}{ll}
					D & \text{if } D\subset C'\\
					C' & \text{if } D\not\subset C'.
				\end{array}\right.
			\end{equation*}
			\item For every $C\leqslant C'$, $p_{C,C'}(D)=D$.
			\item If $C$ and $C'$ are not related ($C\nleqslant C'$ and $C'\nleqslant C$), then we know that there exists $C''\geqslant C,C'$ (because $2^X_F$ with the subset relation is a directed set), and then, $p_{C,C''}(D)=D$ so 
			\begin{equation*}
				p_{C',C''}(D)= \left\{
				\begin{array}{ll}
					D & \text{if } D\subset C'\\
					C' & \text{if } D\not\subset C',
				\end{array}\right.
			\end{equation*}
			and hence 
			\begin{equation*}
				\comp{A}{C}= \left\{
				\begin{array}{ll}
					D & \text{if } D\subset C\\
					C & \text{if } D\not\subset C,
				\end{array}\right.
			\end{equation*}
			so $A=h(D)$.
		\end{itemize}
		\item $h$ is injective.
		Let $C\neq D$ two points of $\hiperf\cup\{X\}$. If one of the two points is $X$, it is clear that the images are different. If both are points of $\hiperf$, then consider $E\geqslant C,D$ and then we have $C=\comp{h(C)}{E}\neq\comp{h(D)}{E}=D$.
		\item $h$ is an open map.
		Let $D\in\hiperf$ a point and $2^D$ its minimal open neighborhood. We claim that $h(2^D)=V\cap\mathcal{X}$, with
		\begin{equation*}
			\comp{V}{C}= \left\{
			\begin{array}{ll}
				2^D & \text{if } C=D\\
				C & \text{if not}.
			\end{array}\right.
		\end{equation*}
		\begin{itemize}
			\item [$\subset$] For every $E\subset D$, we have $\comp{h(E)}{D}=E\in 2^D$.
			\item [$\supset$] Let $A$ be a point of the intersection $V\mathcal{X}$. Then $\comp{A}{D}=B\subset D$. Because of the surjectivity we have $A=h(B)$, with $B\in 2^D$.
		\end{itemize}
		Hence $h$ is open since $V\cap\mathcal{X}$ is open.
	\end{enumerate}
We conclude that $h$, being a well defined, continuous, bijective and open map, is a homeomorphism between the Alexandroff extension of $\hiperf$ and the inverse limit $\mathcal{X}$.
\end{proof}

\section{Embeddings into Alexandroff hyperspaces}\label{sec:embeddings}
Now we define some embeddings into our hyperspaces that allow us to find universal spaces for Alexandroff and finite $T_0$ topological spaces.

Let us recall a result, namely Proposition 2.3 in \cite{MGupper}, about embeddings of a space into its hyperspace with the upper semifinite topology:
\begin{prop}Let $X$ be a Tychonov space. The map $\phi:X\rightarrow 2^X_u$ given by $\phi(x)=\{x\}$ is a topological embedding. Moreover $\phi(X)$ (called the canonical copy of $X$) is dense in $2^X_u$.
\end{prop}
Note that the hyperspaces used in this proposition have the upper semifinite topology given by the topology of $X$ (in contrast to our case, in which they have the discrete topology). This is used to stablish the quoted embedding. We would like to embed topological spaces in the hyperspaces $\hiper$ and $\hiperf$. It is obvious that the same map is not useful here. In fact, we have the following anti-embeddability result:
\begin{prop}
	Let $X$ be a topological space. Then, the map $\phi:X\rightarrow\hiperf$, defined by $\phi(x)=\{x\}$, is continuous if and only if $X$  has the discrete topology.
\end{prop}
\begin{proof}
	If it is continuous then, for every $x\in X$, there exists a neighborhood of $x$, say $U$, such that $\phi(U)\subset2^{\{x\}}=\{x\}$. If $y$ is another point of $X$ lying in $U$, then its image would be $x$, but this is not possible. Then, $U=\{x\}$ is open in $X$ so it is discrete.
\end{proof}
We will need to find a new class of spaces and a different map. Every subspace of $\hiper$ is a $T_0$ Alexandroff space, so the following proposition is natural.
\begin{prop}\label{prop:alexembeddings}
	For every $T_0$ Alexandroff space $X$, there exists a topological embedding $\rho:X\rightarrow\hiper$. If the space $X$ is also locally finite, then the embedding is into $\hiperf$. Moreover, the embedding is as an open subset if and only if $X$ has the discrete topology.
\end{prop}
\begin{proof}
	We define the map
	\mapeo{\rho}{X}{\hiper}{x}{B_x.}
	It is obviously well defined. If $X$ is locally finite, we know by Proposition \ref{teo:localesalexandroff} that the minimal neighborhoods are finite so then, actually, the image is in $\hiperf$. It is a continuous map, because 
	\begin{equation}
		\label{relbases}
		x\leqslant y\Longleftrightarrow B_x\subset B_y\Longleftrightarrow \rho(x)\leqslant\rho(y).
	\end{equation}
	It is injective, because, for two different points $x\neq y$ in $X$, since $X$ is $T_0$, there exists an open neighborhood of one not containing the other, let us say $x\in U\not\ni y$. Then, $x\in B_x\not\ni y$ so $B_x\neq B_y$ and hence $\rho(x)\neq\rho(y)$. It remains to show that the map restricted to its image $\rho:X\rightarrow\rho(X)$ is a homeomorphism. But this is trivial because of relation (\ref{relbases}).
\end{proof}
As we know, the hyperspaces $\hiper$ and $\hiperf$ are determined by the cardinal of the space $X$. So, we can generalize this embeddings a little bit, obtaining universal spaces for alexandroff $T_0$ spaces. A topological space $\mathcal{U}$ is called \emph{universal} for a class $\mathcal{P}$ when it is possible to embed every topological space $X$ of the class $\mathcal{P}$ into it as a topological copy, that is, there exist a continuous map $\tau:X\rightarrow\mathcal{U}$ such that $\tau(X)$ is homeomorphic to $X$.
\begin{prop}
	Let $X$ be any set. The hyperspace $\hiper$ is universal for every $T_0$ Alexandroff space $Y$ with $\text{card}(Y)\leqslant\text{card}(X)$. If $Y$ is locally finite, then the embedding is into $\hiperf$.
\end{prop}
\begin{proof}
	Consider a bijection of $Y$ with a subset of $X$, $\alpha:Y\rightarrow Z\subset X$. We define the map 
	\mapeo{\rho}{Y}{\hiper}{y}{\conjunto{\alpha(y_i): y_i\in B_y}.}
	This map is shown to be well defined, continuous, injective and a homeomorphism if considered onto its image in the same way as in the previous proposition, because relation (\ref{relbases}) holds again.
\end{proof}
\begin{ej}
	We really need the $T_0$ condition. For example, the finite space $X=\{1,2,3\}$, with open sets $\tau=\{\{1\},\{1,2\},\{1,2,3\}\}$, is not $T_0$ and the map $\rho:X\rightarrow\hipernf$ sending $x$ to $B_x$ is not injective.
\end{ej}
\begin{obs}
	What we are really doing is to consider every \textsc{poset} as a family of subsets and inclusions, which is quite natural.
\end{obs}
As an immediate consequence of the last proposition, we have that the power of finite subsets of the natural numbers universal space for $T_0$ finite spaces.
\begin{cor}\label{cor:finiteuniversal}
	The space $\hipernf$ is universal for every $T_0$ finite topological space.
\end{cor}
Recall, from \cite{Sfinite} (see also \cite{Mfinitetopological}) that every finite space is homotopic to a finite $T_0$ space. So we can generalize the universality if we just want to consider the homotopy type of the space.
\begin{cor}\label{cor:finiteuniversal2}
	The space $\hipernf$ is universal up to homotopy type for every finite topological space.
\end{cor}
Using the Alexandroff-McCord functors, we can actually embed every $T_0$ Alexandroff space up to weak homotopy equivalence. Recall the Alexandroff-McCord correspondence from \cite{Msingular}, in which simplicial complexes are related with Alexandroff $T_0$ spaces. On the one hand, for any Alexandroff space space $X$, we define $\mathcal{K}(X)$ as an abstract simplicial complex with vertex set $X$ and simplices the finite totally ordered subsets $x_0\leqslant\ldots\leqslant x_s$ of the poset $X$. On the other hand, given a simplicial complex $K$, we define a Alexandroff $T_0$ space $\mathcal{X}(K)$ with points the simplices of $K$ ordered by inclusion. This correspondence is shown to preserve homotopy and singular homology groups.
\begin{prop}\label{prop:weakly}
	For every $T_0$ Alexandroff space $X$, there is an embedding $\varphi:Y\rightarrow\hiper$, where $Y$ is a topological space weakly homotopically equivalent to $X$ and $\varphi(Y)$ is open in $\hiper$. If $X$ is locally finite, the embedding is into $\hiperf$.
\end{prop}
\begin{proof}
	We define $Y=\mathcal{X}(\mathcal{K}(X))$ and the map $\varphi:Y\rightarrow\hiper$ is just the identity map. It is obvious that, every element of $Y$ belongs to $\hiper$. The continuity, bijectivity onto the image and the continuity of the inverse are also trivial. We just need to show that $\varphi(Y)$ is open in $\hiper$. Let $C$ be a point of $\varphi(Y)$. Then $C\in2^C\subset\varphi(Y)$, because, for every $D\in2^C$ we have $D\in\mathcal{X}(\mathcal{K}(X))$.
\end{proof}
As a direct corollary, we obtain an embedding of the Alexandroff space associated to every simplicial complex.
\begin{cor}\label{cor:simplicalneighborhood}
	Let $K$ be a simplicial complex. Then, there exists an embedding of $\mathcal{X}(K)$ as an open subset of $2^V_u$, where $V$ is the discrete set of vertices. The embedded copy contains the set of singletons of vertices $2^V_1$. If the simplicial complex is finite, the embedding is into $2^V_f$.
\end{cor}

\section{The simplicial neighborhoods category}\label{sec:simplicial}
Let $K$ be a simplicial complex with vertex set $V$. As we saw in Corollary \ref{cor:simplicalneighborhood}, we can identify the simplicial complex with an open subspace $U=\mathcal{X}(K)\subset 2^V_f$ of a hyperspace, such that it contains a canonical copy of the vertex set, that is, $2^V_1=\conjunto{\{v_1\},\ldots,\{v_n\}}\subset U$. We will say that $U$ satisfying these conditions is a \emph{simplicial neighborhood} of the vertex set $2^V_1$. 

We will review some properties and examples of simplicial complexes from this point of view. First, we define some characteristics of simplicial neighborhoods adapting the corresponding ones for the simplicial complexes, taking the concepts from \cite{Salgebraic}. Let $U\subset 2^V_f$ be a simplicial neighborhood. We will say that $U$ is \emph{locally finite} if, for every $C\in U\cap 2^V_1$, the closure $\overline{\{C\}}^U$ is finite. Given any simplicial neighborhood $U\subset 2^V_f$, we will say that $C\in U$ is a \emph{$q$-simplex} if $C\in 2^V_{q+1}\backslash2^V_{q}$. We also say that $C$ has \emph{dimension} $q$. For any two elements $C,D\in U$ satisfying $C\leqslant D$, we say that $C$ is a \emph{face} of $D$ and a \emph{proper face} if $C\neq D$. Moreover, if $C$ has dimension $q$, we say that $C$ is a \emph{$q$-face} of $D$.

\begin{ej} We reformulate some properties of the list of properties for simplicial complexes from \cite{Salgebraic}.
	\begin{enumerate}[noitemsep,nolistsep]
		\item The empty set $\emptyset$ is a simplicial neighborhod.
		\item For every set $V$, $2^V_f$ is a simplicial neighborhood.
		\item Let $C$ be a point of a simplicial neighborhood. Its set of faces, $\overline{C}=\conjunto{D\in U:D\leqslant C}$, is a simplicial neighborhood, because $\overline{C}=2^C$.
		\item For $C\in U$, the set of proper faces of $C$, $\dot{C}=\conjunto{D\in U:D\lneqq C}$ is a simplicial neighborhood. This is so, because $2^C\backslash\{C\}=\bigcup_{D\lneqq C} 2^D$ is open.
		\item For every simplicial neighborhood $U\subset 2^V_f$, its $q$-dimensional skeleton $U_q=U\cap 2^V_{q+1}$, being an intersection of open sets, is a simplicial neighborhood.
		\item Let $X$ be any set. Consider a family $\mathcal{W}=\{W_{\alpha}\}$ of subsets $W_{\alpha}\subset X$. The nerve $\mathcal{N}(\mathcal{W})$ of $\mathcal{W}$,  is the simplicial neighborhood of $2^{\mathcal{W}}_f$ defined by
		$$\{W_{\alpha_0},\ldots,W_{\alpha_q}\}\in\mathcal{N}(\mathcal{W})\Longleftrightarrow W_{\alpha_0}\cap\ldots\cap W_{\alpha_q}\neq\emptyset.$$ It is an open set since every point $\{W_{\alpha_0},\ldots,W_{\alpha_q}\}\in\mathcal{N}(\mathcal{W})$ has an open neighborhood $2^{\{W_{\alpha_0},\ldots,W_{\alpha_q}\}} \subset\mathcal{N}(\mathcal{W})$.
	\end{enumerate}
\end{ej}
We can define a notion of dimension exactly in the same way it is defined for simplicial complexes. Let $U\subset 2^V_f$. We will say that $U$ has \emph{dimension} $0$ if $U=\emptyset$, \emph{dimension} $n$ if $U\subset 2^V_{n+1}\backslash2^V_{n}$ and \emph{dimension} $\infty$ if $U\not\subset 2^V_n$ for every $\todon$. We will say that $U$ is a \emph{finite simplicial neighborhood} if it is finite as a set. Given any simplicial neighborhood $U\subset 2^V_f$, a \emph{simplicial subneighborhood} $W\subset U$ is just an open space contained in $U$. We will say that a simplicial subneighborhood $W\subset U$ is \emph{full} if, for every $C\in U$ satisfying $2^C\cap 2^V_1\subset W$, the closures in both spaces are the same, $\overline{\{C\}}^U=\overline{\{C\}}^W$.

\begin{ej}More examples from \cite{Salgebraic}.
	\begin{enumerate}[resume,noitemsep,nolistsep]
		\item For every $q\in\mathbb{N}$, the $q$-skeleton $U_q$ is a simplicial subneighborhood of $U\subset 2^V_f$. For $p\leqslant q$, $U_p$ is a simplicial subneighborhood of $U_q$.
		\item For every $C$ in a simplicial neighborhood $U\subset 2^V_f$, we have that $\dot{C}\subset\overline{C}\subset U$ are simplicial subneighborhoods.
		\item Consider a family $\{U_j\}_{j\in J}$ of simplicial subneighborhoods of a simplicial neighborhood $U$. Then the union $\bigcup_{j\in J} U_j$ and the intersection $\bigcap_{j\in J} U_j$ are simplicial subneighborhoods of $U$.
		\item For $A\subset X$, and $\mathcal{W}=\{W_{\alpha}\}$ with $W_{\alpha}\subset X$, the nerve of $A$, defined as $\mathcal{N}_A(\mathcal{W})=\mathcal{N}(\mathcal{W})\cap 2^A$, is a simplicial subneighborhood of $\mathcal{N}(\mathcal{W})$.
	\end{enumerate}
\end{ej}
Given a simplical map between simplicial complexes, $\varphi:K_1\longrightarrow K_2$, we can define a map between the corresponding simplicial neighborhoods $\psi:U_1\longrightarrow U_2$ as the map $\psi=\mathcal{X}(\varphi)$. That is, the map between the simplicial neighborhoods is defined as \mapeo{\psi}{U_1}{U_2}{\{v_0,\ldots,v_n\}}{\{\varphi(v_0),\ldots,\varphi(v_n)\}} This map is obviously continuous (as seen in \cite{Mfinitecomplexes}). Moreover, it is an open map, because, for every $C\in U_1$, $\psi(2^C)=2^{\psi(C)}$. So, it is evident that the application $\mathbf{\mathcal{X}}$ is a covariant functor between the category of simplicial complexes and simplicial maps and the category of simplicial neighborhoods of hyperspaces (with the upper semifinite topology) and continuous and open maps between them. Hence, we can define a \emph{hypersimplicial map} between simplicial neighborhoods as a continuous and open map between them.

We will define now an inverse of this functor. 
For every simplicial neighborhood $U\subset2^V_f$, we define a simplicial complex $K=\mathcal{Y}(U)$ as follows. The vertices of $K$ are the points of $U\cap 2^V_1$ and the simplices are the points of $U\backslash2^V_1$. In this way, every face of a simplex $\tau\subset\sigma$ is a simplex, since, as points in $U$, if $\sigma\in U$, then $2^{\sigma}\subset U$, because $U$ is open. For every hypersimplicial map between simplical neighborhoods $\psi:U_1\rightarrow U_2$, we define a simplicial map between the simplicial complexes $\mathcal{Y}(\psi):\mathcal{Y}(U_1)\rightarrow\mathcal{Y}(U_2)$ as $\mathcal{Y}(\psi)(v_1)=v_2$, where $\psi(\{v_1\})=\{v_2\}$. It is a simplicial map: For every $C\in U$, we have that $\{C\}$ is open in $U$ if and only if $C$ consists of only one element $C=\{v\}$ (because its minimal neighborhood $2^C$ has to be only $C$). So, it is clear that $\mathcal{Y}(\psi)$ sends vertices to vertices. Moreover, $\psi$ is continuous, so $C\subset D$ implies $\psi(C)\subset\psi(D)$, and that ensures that the induced map $\mathcal{Y}(\psi)$ sends simplices to simplices. These functors are mutually inverse, so we have the following.
\begin{cor}\label{cor:equivalent}
	The category of simplicial neighborhoods and hypersimplicial maps is equivalent to the category of abstract simplicial complexes and simplicial maps.
\end{cor}

\paragraph{Universality for shape properties}
We can use subsets of a given compact metric space, lying in the hyperspace, to determine the shape properties of the space that are encoded in some way in the hyperspace. For details about shape theory, we recommend \cite{MSshape}. Recall, from \cite{MGhomotopical}, that for every compact metric space \cms we can define, for every $\epsilon>0$, the subsets of $X$ consisting of 
$$U_{\epsilon}=\conjunto{C\subset X \text{ closed}:\diam(C)<\epsilon}\subset 2^X_u.$$
In that paper it is shown that the family $\conjunto{U_{\epsilon}}_{\epsilon>0}$ is a base of open neighborhoods of the canonical copy of $X$ for the topology of $2^X_u$. We can define these sets still in the hyperspaces $\hiper$ (although the closed condition is unnecessary) and they are also open, because for every $\epsilon>0$ and every $C\in U_{\epsilon}$, we have $2^C\subset U_{\epsilon}$. We can use these sets to determine the shape of $X$ and hence the \v{C}ech homology and every shape property.
\begin{prop}\label{prop:univshape}
	Let \cms be a compact metric space. There exists an inverse system of subspaces of $\hiperf$ such that their McCord associated inverse system is an HPol expansion of $X$. 
\end{prop}
\begin{proof}
	We consider, for every $\epsilon>0$, the intersection $$U_{\epsilon}^f=U_{\epsilon}\cap\hiperf,$$ which is open, $T_0$ and Alexandroff. For every pair $\epsilon'<\epsilon$ we have the inclusion map $$i_{\epsilon,\epsilon'}:U_{\epsilon'}^f\longrightarrow U_{\epsilon}^f.$$ So, we can consider the inverse system of the McCord associated polyhedra and maps, $$\conjunto{|\mathcal{K}(U_{\epsilon}^f)|,|\mathcal{K}(i_{\epsilon,\epsilon'})|,\mathbb{R}}.$$
	As it is shown in Corollary 7 of \cite{MCLepsilon} (for a finite set of vertices but the same proof extends to an infinite set of vertices), the simplicial complex $\mathcal{K}(U_{\epsilon}^f)$ is isomorphic to the barycentric subdivision of the Vietoris-Rips complex $\mathcal{R}_{\epsilon}(X)$, so their realizations are homeomorphic. So, this inverse system is isomorphic to the Vietoris system (see \cite{MSshape} for a description), which is an HPol-expansion of $X$.
\end{proof}
The same construction can be done for every different metric (generating the same topology or not) in the set $X$. So, in that sense, $\hiperf$ is universal for every shape property of every possible metric given over $X$. This should be extendable to every topology on the set $X$ using refinements.

Note that we have encoded all the shape information of a compact metric space (actually all the shape information of every possible metric on the set that makes it compact) in terms of the category of simplicial neighborhoods and hypersimplicial maps.

\appendix
\section{Proof of Theorem \ref{teo:inverseextension} in the countable case}\label{app:countable}
\begin{teo}
	The Alexandroff extension of $\hipernf$ is homeomorphic to an inverse limit of an inverse sequence of finite spaces.
\end{teo}
\begin{proof}
	We should think about this space as an countable cone over the point $\{1\}$. This allows us to understand what follows. The natural numbers are totally ordered. In this case, there is a sequence of ordered-by-inclusion open sets (of the basis) $$B_{\{1\}}\subset B_{\{1,2\}}\subset B_{\{1,2,3\}}\subset\ldots$$ such that $$\bigcup_{s=2,\ldots,\infty}B_{\{1,\ldots,s\}}=\hipernf.$$ Despite of this, they are not a basis: for example, the point $\{1,4\}\in B_{\{1,2,4\}}$ but it is imposible to find an $s$ such that $\{1,4\}\in B_{\{1,\ldots,s\}}\subset B_{\{1,2,4\}}$. We want to construct an inverse sequence in terms of this ordered chain. We can define a natural map (a kind of ``collapse'' -this is not formal!) from every element of the chain to a lower one. For every $\todon$, define a map $p_{n,n+1}:2^{\{1,\ldots,n,n+1\}}\rightarrow 2^{\{1,\ldots,n\}}$ as 
	\begin{equation*}
		p_{n,n+1} = \left\{
		\begin{array}{lr}
			C & \text{if } n+1\notin C,\\
			\{1,\ldots,n\} & \text{if } n+1\in C.
		\end{array}\right.
	\end{equation*}
	This map is continuous: Suppose we have a pair of points of $2^{\{1,\ldots,n,n+1\}}$, namely $C\subset D$. Then, there are three different cases:
	\begin{itemize}
		\item If $n+1\in C\subset D$, then $p_{n,n+1}(C)=p_{n,n+1}(D)=\{1,\ldots,n\}$.
		\item If $n+1\notin C,D$, then $p_{n,n+1}(C)=C\subset D=p_{n,n+1}(D)$.
		\item If $n+1\notin C$ but $n+1\in D$, then $p_{n,n+1}(C)=C\subset\{1,\ldots,n\}=p_{n,n+1}(D)$.
	\end{itemize}
	So, $p_{n,n+1}$ is continuous for every $n\in\todon$.
	\begin{table}
		$$\xymatrix@R=0.05cm@C=1.2cm{2^{\{1\}} & 2^{\{1,2\}}\ar[l]_{p_{1,2}} & 2^{\{1,2,3\}}\ar[l]_{p_{2,3}} & 2^{\{1,2,3,4\}}\ar[l]_{p_{3,4}} & \ldots \ar[l]_{p_{4,5}}\\
			\{1\} & \{1\}\ar[l] & \{1\}\ar[l] & \{1\}\ar[l] & \ldots \ar[l]\\
			& \{2\}\ar[ul] & \{2\}\ar[l] & \{2\}\ar[l] & \ldots \ar[l]\\
			& \{1,2\}\ar[uul] & \{1,2\}\ar[l] & \{1,2\}\ar[l] & \ldots \ar[l]\\
			&                 & \{3\}\ar[ul] & \{3\}\ar[l] & \ldots \ar[l]\\
			&                 & \{1,3\}\ar[uul] & \{1,3\}\ar[l] & \ldots \ar[l]\\
			&                 & \{2,3\}\ar[uuul] & \{2,3\}\ar[l] & \ldots \ar[l]\\
			&                 & \{1,2,3\}\ar[uuuul] & \{1,2,3\}\ar[l] & \ldots \ar[l]\\
			&                 &                     & \{4\},\{1,4\},\ldots\ar[ul] & \ldots \ar[l]      
		}$$
		\caption{Visualization of the inverse limit of $\conjunto{2^{\{1,\ldots,n\}},p_{n,n+1}}$.}
		\label{tab:limnat} 
	\end{table}	
	Now it makes sense to ask what is the inverse limit of the inverse sequence $$\conjunto{2^{\{1,\ldots,n\}},p_{n,n+1}}$$
	Table \ref{tab:limnat} shows a visualization of the first elements of the sequence. We see that each element $\{a_1,\ldots,a_s\}$ (suppose ordered) of $\hipernf$ is represented in the inverse limit as an element that begins at the $a_s$-th term of the sequence. But there is an element of the inverse sequence that is not any of the previously described, namely $$(\{1\},\{1,2\},\{1,2,3\},\{1,2,3,4\},\ldots).$$ So, we claim that the inverse limit $\mathcal{N}$ of the inverse sequence $$\conjunto{2^{\{1,\ldots,n\}},p_{n,n+1}}$$ is homeomorphic to the subspace $\hipernf\cup\{\mathbb{N}\}\subset\hiper$.
	For every $C\in\hipernf$, let us write the maximum of its elements as $m(C)=\max\conjunto{c_i\in C}$.
	Define the following map, $h:\hipernf\cup\{\mathbb{N}\}\rightarrow\mathcal{N}$, as 
	\begin{align}
		h(\mathbb{N})&=\left(\{1\},\{1,2\},\{1,2,3\},\ldots\right) ,\nonumber\\
		h(C)&=\left(\{1\},\{1,2\},\ldots,\{1,2,\ldots,m(C)-1\},C,C,\ldots\right), \texfor{for every} C\in\hipernf.\nonumber
	\end{align}
	We will show that this map is a homeomorphism between the two spaces.
	\begin{enumerate}
		\item $h$ is well defined: It is obvious that $h(\mathbb{N})\in\mathcal{N}$. And, for every $C\in\hipernf$, $h(C)\in\mathcal{N}$, since
		\begin{itemize}
			\item for $n\leqslant m(C)-2$, $p_{n,n+1}(\{1,\ldots,n,n+1\})=\{1,\ldots,n\}$,
			\item $p_{m(C)-1,m(C)}(C)=\{1,2,\ldots,m(C)-1\}$ and
			\item for $n\geqslant m(C)$, $p_{n,n+1}(C)=C$.
		\end{itemize}
		\item $h$ is injective, as easily checked from the definition.
		\item $h$ is surjective: Consider $\left(C_1,C_2,\ldots,\right)\in\mathcal{N}$. Two possibilities:
		\begin{itemize}
			\item If, for every $n\in\mathbb{N}$, $C_n\neq C_{n+1}$, then $C_n=\{1,2,\ldots,n\}$ for every $n\in\mathbb{N}$ and $h(\mathbb{N})=\left(C_1,C_2,\ldots\right)$.
			\item If there exists $n_0\in\mathbb{N}$, such that, $C_{n_0}=C_{n_0+1}$, let us suppose that it is the minimum satisfying this condition and then: For every $n<n_0$ we have that $C_n\neq C_{n+1}$, so $C_n=\{1,\ldots,n\}$. For every $n>n_0$, $n\notin C_n$, so $C_{n+1}=C_n$. In this case, $h(C_{n_0})=\left(C_1,C_2,\ldots\right)$.
		\end{itemize}
		\item $h$ is continuous: Let us divide the proof in two cases.
		\begin{itemize}
			\item Every open neighborhood $U$ of $h(\mathbb{N})$ must be $$U=\left(2^{\{1\}}\times2^{\{1,2\}}\times\ldots\right)\cap\mathcal{N}=\mathcal{N}$$ and $h\parentesis{\hipernf\cup\{\mathbb{N}\}}\subset U$.
			\item For every $C\in\hipernf$, consider an open neighborhood $W$ of $h(C)$ in $\mathcal{N}$. Then $h(C)\in V\subset W,$ with $$V=\parentesis{2^{\{1\}}\times2^{\{1,2\}}\times\ldots\times2^{\{1,\ldots,m(C)\}}\times2^C\times2^{\{1,\ldots,m(C)+2\}}\times\ldots}\cap\mathcal{N}.$$ The open neighborhood $2^C$ of $C$ satisfies that $h(2^C)\subset V$, since, for every $D\in2^C$, $h(D)\in V$, because $m(D)\leqslant m(C)$.
		\end{itemize}
		\item $h$ is an open map: It sends every basic open set to an open set. Namely, $h(\hipernf\cup\{\mathbb{N}\})=\mathcal{N}$ and, for every $C\in\hipernf$, $$h(2^C)=\parentesis{2^{\{1\}}\times2^{\{1,2\}}\times\ldots\times2^{\{1,\ldots,m(C)\}}\times2^C\times2^{\{1,\ldots,m(C)+2\}}\times\ldots}\cap\mathcal{N}.$$
		\begin{itemize}
			\item[$\subset$] For every $D\in2^C$, $h(D)\in\mathcal{N}$ by definition, and $$h(D)\in\parentesis{2^{\{1\}}\times2^{\{1,2\}}\times\ldots\times2^{\{1,\ldots,m(C)\}}\times2^C\times2^{\{1,\ldots,m(C)+2\}}\times\ldots},$$ because, for $n=2,\ldots,m(D)-1$, $\{1,\ldots,n\}\in2^{\{1,\ldots,n\}}$, for $n=m(D),\ldots,m(C)$, $D\in2^{\{1,\ldots,n\}}$, $D\in2^C$ (case $n=m(C)+1$) and, for $n\geqslant m(C)+2$, $D\in2^{\{1,\ldots,n\}}$.
			\item[$\supset$] Any element $A$ of this set is of the form 
			\begin{align*}A=(p_{1,m(C)+1}(D),p_{2,m(C)+1}(D),
				\ldots,&p_{m(C),m(C)+1}(D),D,\\
				&p^{-1}_{m(C)+1,m(C)+2}(D),\ldots),\end{align*} 
			for some $D\in2^C$. But, it is easy to see that $$A=\parentesis{\{1\},\{1,2\},\ldots,\{1,\ldots,m(D)-1\},D,D,\ldots}=h(D).$$
		\end{itemize}
	\end{enumerate}
	The map $h$, as a continuous open bijection, is a homeomorphism.
\end{proof}

\paragraph{Acknowledgements}
The author is indebted to his thesis supervisor M.A. Morón for his indispensable help and inspiration in exploring the wonderful worlds of non-Hausdorff topologies.

\paragraph{Funding}
This work has been partially supported by the research project PGC2018-098321-B-I00 (MICINN) and the FPI Grant BES-2010-033740 of the project MTM2009-07030 (MICINN).

\addcontentsline{toc}{chapter}{References}
\bibliographystyle{siam}
\bibliography{Universal}
	
\vspace{1cm}
\noindent\textsc{Diego Mondéjar}\\
Departamento de Matemáticas\\
CUNEF Universidad\\
\texttt{diego.mondejar@cunef.edu}

\end{document}